\documentclass{article}

\usepackage{amsthm,amssymb,amscd,amsmath}

\newtheorem{definition}{Definition}[section]
\newtheorem{proposition}[definition]{Proposition}
\newtheorem{theorem}[definition]{Theorem}
\newtheorem{example}[definition]{Example}

\newcommand{\Z}{\mathbb{Z}}
\newcommand{\D}{\mathcal{D}}
\newcommand{\E}{\mathcal{E}}

\newcommand{\A}{\mathsf{A}}
\newcommand{\B}{\mathsf{B}}
\newcommand{\LL}{\mathsf{L}}
\newcommand{\TT}{\mathsf{T}}
\newcommand{\MM}{\mathsf{M}}
\newcommand{\I}{\mathsf{I}}
\newcommand{\NN}{\mathsf{N}}
\newcommand{\JJ}{\mathsf{J}}
\newcommand{\ZZ}{\mathsf{Z}}
\newcommand{\im}{{\rm im}\,}
\newcommand{\di}{{\rm dim}}

\begin{document}

\title{The algebraic structure of quantity calculus}
\author{\'Alvaro P. Raposo\thanks{e-mail:\texttt{alvaro.p.raposo@upm.es}} \\ Department of Applied Mathematics \\ Universidad Polit\'ecnica de Madrid\\ Av. Juan de Herrera, 6, Madrid, 28040, Spain}
\date{June 29, 2016}
\maketitle

\begin{abstract}
An algebraic structure underlying the quantity calculus is proposed consisting in an algebraic fiber bundle, that is, a base structure which is a free Abelian group together with fibers which are one dimensional vector spaces, all of them bound by algebraic restrictions.  Subspaces, tensor product and quotient spaces are considered as well as homomorphisms to end with a classification theorem of these structures.  The new structure provides an axiomatic foundation for quantity calculus and gives complete justification within its framework of the way that quantity calculus is actually performed.  It is hoped that this exposition helps to clarify the role of the interviening concepts of quantity, quantity value, quantity dimension and their relation with a system of units, particularly, the SI.
\end{abstract}




\section{Introduction}\label{sec:introduction}
Quantity calculus is the algebra of the operations performed among quantities, which are three: product of quantities, product of a number times a quantity and addition of quantities, with the particularity in the latter case that only quantities of the same kind may be added. There have been, and still are, intense debates around the few concepts involved in the previous sentence: quantity, quantity calculus, quantities of the same kind, and also quantity dimension and quantities of dimension one (formerly referred to as dimensionless quantities).  All of these terms are defined in the International Vocabulary of Metrology (VIM) \cite{VIM3}, but the definitions are not always as clear as desirable and this may be caused by the lack of a formalism supporting quantity calculus.

Actually, the rules of quantity calculus are well handled by scientist and technicians and the mere existence of the International System of Quantities (ISQ) and the International System of Units (SI) is a demonstration of the high degree of sofistication attained in its use.  Despite of it, this algebra has not yet been described from a simple set of axioms from which the rest of properties are implied, as was acknowledged by de Boer in his classical paper of 1994 \cite{Boe94}.  Let us examine why.

One of the starting points of quantity calculus is located in Maxwell's book on electricity and magnetism \cite{Max73} where he refers to quantities as being expressed by the product of a number and a unit, the latter being a standard quantity of reference of the same kind as the quantity considered.  This is usually written for the quantity $q$ as
\begin{equation}\label{eq:Maxwell}
q = \{q\} \cdot [q],
\end{equation}
where $[q]$ stands for the unit and $\{q\}$ is the number of times $q$ comprises the unit.  In fact, this splitting of a quantity in a numerical value and a reference is the definition of the concept quantity in the VIM \cite{VIM3}, but notice that such a definition is void, for it defines a quantity $q$ as a property that can be expressed as a number times another quantity, $[q]$.  Equation (\ref{eq:Maxwell}) is only a way in which a quantity may be expressed, but not a definition of the concept.

It is de Boer in his paper of 1994 \cite{Boe94} about the history of quantity calculus who points that  equation (\ref{eq:Maxwell}) has been taken as the definition of the notion of quantity due to a lack of a formalism.  At the end of his paper, de Boer describes the state of the art in the problem of the formalization of the algebraic structure of quantity calculus in the terms which are summarized here:
\begin{itemize}
\item Multiplicative properties: The set of quantities has a product operation with the properties of associativity, commutativity, identity and inverse for nonzero quantities.  This gives the set the algebraic structure of a monoid.
\item Quantities of the same kind: The set of quantities is partitioned into classes of quantities of the same kind. Only inside each class addition is allowed and the product of a number times a quantity result in a quantity in the same class.  These two operations satisfy the axioms of a vector space, so each class is a vector space.  Moreover, each of them is of (vector space) dimension one, for any quantity can be expressed as a number times a nonzero quantity selected as a standard for that class, i.e., the unit in this class.
\item Several isomorphic groups:  There are several group structures recognized in this framework, all of them isomorphic.  First, the equivalence classes can be multiplied considering that the product of two quantities of any kind gives a quantity of another kind.  This product of classes result in a group structure for the set of equivalence classes.  A system of units is a choice of a unit, that is, a nonzero quantity, in each equivalence class; the system is coherent if the product of units result in another unit of the system. If that is the case, the set of units is also a group under the product, and it is isomorphic with the group of equivalence classes.  Additionally there is the concept of dimension, which is defined as the quality shared by quantities of the same kind.  Then there is an obvious product of dimensions and it gives the set a group structure, isomorphic with the other two groups mentioned.
\end{itemize}
De Boer also points that these elements are lacking of a formalization which gathers all of them together in a single algebraic structure defined by a simple list of axioms.

In the review article by Foster \cite{Fos10} we find such a list of four axioms for the ISQ:
\begin{enumerate}
\item The dimension of a quantity is the product of the powers of factors corresponding to the base quantities of the system.
\item A quantity $q$ is the product of a numerical value $\{q\}$ and a unit $[q]$.
\item The product of two quantities is the products of their numerical values and units
\[
q_1q_2 = \{q_1\}\{q_2\} \, [q_1][q_2],
\]
where $\{q_1\}\{q_2\}=\{q_1q_2\}$ and $[q_1][q_2] = [q_1q_2]$.
\item The quotient of two quantities is the quotients of their numerical values and units
\[
q_1/q_2 = \{q_1\}/\{q_2\} \, [q_1]/[q_2],
\]
where $\{q_1\}/\{q_2\}=\{q_1/q_2\}$ and $[q_1]/[q_2] = [q_1/q_2]$.
\end{enumerate}
Notwithstanding their correctness, these four statements are not a set of axioms.  The first one lacks of previous definitions of the realm in which the concepts of quantity and dimensions belong.  The second one looks more a definition than an axiom and, as commented before, it is actually a void definition.  The third and four show the same flaws as the second: they look as definitions and they rely the definition of an operation of quantities on the same operation performed on another quantities (the units).  In addition, in every algebraic system with a product operation and an identity element, the quotient is not a different operation, but the product with inverses and, thus, the problem of taking quotients is translated to the problem of the existence of inverses.  Finally, this set of axioms does not mention the addition operation.  Therefore, this attemp does not answer de Boer's question.

There exists a succesful attempt for such a formalism proposed by Drobot \cite{Dro53} as early as 1953, and improved by Whitney \cite{Whi68b}.  Their goal was to give a sound foundation to dimensional analysis, in particular the celebrated Buckingham's Pi Theorem.  The key to this formalism is that, in a coherent system of units, if $q_1, \dots, q_k$ are units, so is $q_1^{\alpha_1} \cdots q_k^{\alpha_k}$.  Now, if the exponents $\alpha_1, \dots, \alpha_k$ are taken in a field, say the rational numbers or the real numbers, then the set of units shows a structure of a vector space written multiplicatively.  Of course, carrying such a structure is a great advantage, for all the tools of linear algebra become available.  However, when addition of quantities and product with real numbers are introduced, the structure gets more complex.  In fact, Carlson \cite{Car79} and Kock \cite{Koc89} explore further this algebraic structure, but restricting themselves to the product operation, so they deal  only with the linear properties.

While Drobot-Whitney's formalism gives a rigorous basis to dimensional analysis, as it is their author's intention, the structure they depict is not quite the algebraic structure that de Boer was pleading for a system of quantities.  The use of exponents $\alpha_1$, \dots, $\alpha_k$ from a field  seems unnecessary.  It is not justified in the properties of actual quantities, for it is a remarkable fact that dimensionful quantities (in contrast with dimensionless quantities) appear in laws of physics with the only operations of product, product by a number and addition between quantities of the same kind.  There is no need for exponentiation with fractional, less real, exponents.  Only quantities of dimension one, pure numbers, appear with noninteger exponents or as arguments of functions such as trigonometric, exponential or logarithmic.  Whitney justifies the use of rational or real exponents on the example of a free falling body from height $h$, where the time $t$ needed to reach the ground is given by $t=\sqrt{2gh}$, with $g$ the acceleration of gravity so, he concludes, to compute the quantity $t$, the square root of the quantity $2gh$, which is dimensionful, must be considered.  It is clear, however, that the square root is taken on a quantity whose dimension is time squared, and that this is always the case with fractional exponents so, if fact, it is artificial to allow a generalized entrance to these exponents.  Another example  which comes to mind is the equation of the curve representing a reversible adiabatic process of an ideal gas in the $p-v$ plane (pressure and volume), which is usually written as $pv^\gamma={\rm constant}$, where $\gamma$ is the adiabatic index of the gas, a noninteger number, in general.  The derivation of this equation shows that the original form of it is $(p/p_0)(v/v_0)^\gamma=1$, where $p_0$ and $v_0$ are the pressure and volume at some point of reference in the curve, so the quantities involved with the noninteger exponent are of dimension one.

A more recent effort in the direction pointed by de Boer is the paper by Krystek \cite{Kry15} in which he describes with full detail the set of dimensions and the group structure it carries, as initiated in the third point above summarizing de Boer's paper, but now with a powerful mathematical approach to describe those isomorphic groups.  His purpose is to acknowledge that the quantities of dimension one are, in fact, those whose dimension is the identity element in this group and, hence, proposes a symbol for this identity dimension, $\ZZ$. In the present paper, the ideas advanced by Krystek are pursued on and expanded.

The goal of this paper is to provide a new abstract algebraic structure, defined on a simple set of axioms, which accout exactly for the properties of quantity calculus.  This algebraic structure is the one underlying any system of quantities, in particular the ISQ, and, thus, it can help to clarify some of the concepts still under debate within the definitions in the VIM or some of the concepts underlying the SI.

A complete description, in mathematical terms, of such an algebraic structure must permit us, first, define with precision and locate appropriately each of the concepts involved, for instance quantity value, numerical value of a quantity, dimension, dimension one, and so on.  Second, justify the expression of a quantity as a number times a reference, that is, equation (\ref{eq:Maxwell}), which no longer should be taken as a definition; moreover, the symbols $\{\cdot\}$ and $[\cdot]$ must be studied and their algebraic properties described.  Third, justify the actual way in which operations among quantities are performed, that is, by means of their expressions like in equation (\ref{eq:Maxwell}) and operating with the numerical values and the units separately.  Fourth, classify the possible systems of quantities from the viewpoint of its algebraic structure, in particular, confirm the classification of the ISQ within this scheme.

Although much advantage is taken from the work of Krystek \cite{Kry15}, a key difference is worth noticing between his approach and the present one.  In Krystek's paper the starting point is a system of quantities from which the group of dimensions is built by means of a quotient.  In the present paper, however, the starting point is the group of dimensions, which is defined in the first place, and which upon it the system of quantities is built.  Despite of it the quotient relations shown by Krystek are also valid in the present scheme, for the latter includes completely the former.

The layout is as follows.  In section~\ref{sec:space} the set of dimensions and its group structure is defined in advance and, on top of it, the definition of space of quantities is built, followed by the consideration of systems of units and the justification of the usual properties of quantity calculus, in particular, the symbols $\{\cdot\}$ and $[\cdot]$ are shown to be maps with suitable algebraic properties. In section~\ref{sec:new-spaces} we define subspaces of a given space, product of spaces and quotient spaces, as tools for building new spaces from old ones, in particular it is shown the latter to be the technique for defining the so called natural units.  In section~\ref{sec:homomorphism} homomorphisms of spaces of quantities are defined and, with them at hand, we are able to compare spaces of quantities, define isomorphic spaces and classify them, giving a complete characterization up to isomorphism.


\section{Group of dimensions and space of quantities}\label{sec:space}

Our object of study is a set $Q$ of quantities and the operations defined within it, that is, a system of quantities as defined by the VIM \cite{VIM3} but, in order to distinguish the abstract algebraic structure considered here from the particular instance considered in the VIM, the term space of quantities will be used here instead of system of quantities.  The elements of $Q$ will be denoted by lowercase latin letters, particularly $q$, $r$, $s$. These quantities, as it is detailed below, can be multiplied and added among them and also multiplied by scalar numbers from a field $F$, whose elements will be denoted by lowercase greek letters, particularly $\alpha$, $\beta$.
\begin{example}\label{ex:sets-of-quantities}
For further reference, we assign symbols to the following spaces of quantities: $Q_{\rm geom}$, the space of quantities of geometry, that is, all the quantities needed to deal with lengths, areas, volumes, angles, etc. $Q_{\rm time}$, the space of quantities to measure time. $Q_{\rm kin}$, the space of quantities of kinematics. $Q_{\rm mech}$, the space of quantities of mechanics. $Q_{\rm phys}$, the space of quantities of physics. $Q_{\rm ISQ}$, the International System of Quantities.
\end{example}
It is worth noticing that in the algebraic structure we are about to define it is meaningless to distinguish the concepts quantity and quantity value, as it is explicitly done in VIM \cite{VIM3} and oftenly discussed \cite{Mar12}.  In fact, within the algebraic structure, the symbol $q$ stands for a quantity as well as the symbol $[q]$ and, thus when we write $q= \{q\} \, [q]$ we identify in the structure both sides of the equality, while in the VIM the left hand side is viewed as the quantity itself and the right hand side as its value.

\subsection{Group of dimensions}\label{subsec:dimensions}
As noticed in the introduction, and in contrast with the opinion of some authors \cite{Eme05}, a main role in the structure is played by the dimension of a quantity.  The importance of dimensions lies in that the dimension of a quantity is an intrinsic property of it, in contrast with its numerical value, which depends on the unit chosen, or the unit itself, which can be changed arbitrarily.  Therefore each quantity must have a firm link with its dimension in the present scheme better than a link with a unit or its numerical value with respect to that unit, despite the latter is what equation (\ref{eq:Maxwell}) suggests.  To that end, let us first define properly the set of dimensions. The properties which characterize this set have been well described in the paper by Krystek \cite{Kry15}, and they are just summarized here: dimensions can be multiplied and show the structure of an Abelian group with two further properties which characterize this group.  First, no element is torsion, for there is no dimensionful quantity which multiplied by itself finitely many times becomes a quantity of dimension one.  Second, it is finitely generated. Therefore we adopt the following definition in which the set of dimensions is referred to at once as the group of dimensions.
\begin{definition}\label{def:group-of-dimensions}
A group of dimensions is a finitely generated free Abelian group.
\end{definition}
In this paper such a group is generally denoted by $\D$ and its elements by uppercase letters in roman sans-serif type such as $\A, \B\dots$ (as stated in the VIM).

The identity element of the group deserves a bit of attention.  It has been recognized \cite{Boe94,Kry15} that it stands for the dimension of the quantities of dimension one.  It is also under debate the name for the quantities with this dimension: dimensionless is being abandoned (righteously, as these quantities do have a dimension) but dimension one, the name now proposed in the VIM, is not well settled \cite{Bro15,Kry15,Mil16}.  The common use in the mathematical theory of groups is to denote the identity element of a group with the symbol 1, which is well understood not to be the number one, but an element of the group.  In this paper the symbol $1_{\D}$ is used, which follows the group theoretical tradition but reminds the reader that the symbol belongs to $\D$ and, thus, it is not a number.

Two properties of finitely generated free Abelian groups are of interest for us \cite{Rot95}. In the first place, there exists the concept of basis: an independent (finite) set of generators.  If $\{\A_1, \dots, \A_k\}$ is such a basis for a group $\D$ then any element $\B$ has a unique expression in terms of it of the form $\B=\A_1^{n_1} \cdots \A_k^{n_k}$, where the exponents $n_1$, \dots, $n_k$ are integer numbers.  The number $k$ of generators of any basis is called the rank of the group, and is a characteristic property of a free Abelian group.  In the second place, such a group is isomorphic with the direct product of $k$ infinite cyclic groups: $\D \cong \langle \A_1 \rangle \times \cdots \times \langle \A_k \rangle$, where $\langle \A_i \rangle = \{\A_i^n\,:\, n \in \Z\}$.
\begin{example}\label{ex:groups-of-dimensions}
The groups of dimensions of the systems of quantities given in example \ref{ex:sets-of-quantities} are, respectively, the following:  $\D_{\rm geom} = \langle \LL \rangle$, the free Abelian group generated by $\LL$, which denotes length. $\D_{\rm time} = \langle \TT \rangle$, generated by $\TT$ (time). $\D_{\rm kin} = \langle \LL, \TT \rangle$, generated by $\LL$ and $\TT$. $\D_{\rm mech} = \langle \LL, \TT, \MM \rangle$, generated by $\LL$, $\TT$ and $\MM$ (mass). $\D_{\rm phys} = \langle \LL, \TT, \MM, \I, \Theta \rangle$, generated by $\LL$, $\TT$, $\MM$, $\I$ (electric current) and $\Theta$ (thermodynamic temperature). $\D_{\rm ISQ} = \langle \LL, \TT, \MM, \I, \Theta, \NN, \JJ \rangle$, generated by $\LL$, $\TT$, $\MM$, $\I$, $\Theta$, $\NN$ (amount of substance) and $\JJ$ (luminous intensity). Of course the dimensions $\LL, \TT, \MM, \I, \Theta, \NN, \JJ$ are considered independent, that is, no one can be obtained as a product of powers of the remaining six.
\end{example}

\subsection{Space of quantities}\label{subsec:space}
The link between a quantity and its dimension is made by means of a map $\di\!\!:Q\to\D$, that is, the dimension of the quantity $q$ is $\di(q)$. This map is obviuosly a surjection. Krystek \cite{Kry15} has denoted this map as $\delta$ in order to distinguish it from the map $\di$ as defined by the VIM.  The difference, as he himself explains, is that $\di$ applies to quantities, while $\delta$ to quantity values.  Since the algebraic structure we are developing does not distinguish between quantity and quantity value, both maps are identified as well in this paper.
\begin{example}\label{ex:dimensions}
If $h$ is Planck's constant then $\di(h)=\LL^2\TT^{-1}\MM$; and $\theta$, the angle at a vertex of a triangle, yields $ \di(\theta)=1_{\D}$; each one in the appropriate setting.
\end{example}

 In order to reflect that the dimension of a product of quantities is the product of the dimensions of the quantities the projection map must be a homomorphism with respect to the product of quantities.

All the quantities with the same dimension, say $\A$, form a set called a fiber, for it can be written as the inverse image of that dimension: $\di^{-1}(\A)$. As the VIM explicitly states, quantities of the same kind belong to the same fiber, while the opposite is not necessarily true.  In each fiber quantities can be added and multiplied by scalars in a field $F$, resulting in quantities of the same dimension.  These operations give the fiber the structure of a vector space (written additively, as usual) over the field $F$.  Moreover, since the comparison of each quantity in the fiber with a reference in the fiber, the unit, yields a single number, as in equation (\ref{eq:Maxwell}), that vector space is one dimensional (the latter in the sense of vector space dimension over $F$).  The field $F$ is usually assumed to be that of the real numbers but so far there is no algebraic reason to restrict the definition to it.  We are now ready to give a formal definition of a space of quantities which takes into account all the elements aforementioned.
\begin{definition}\label{def:space-of-quantities}
A space of quantities with group of dimensions $\D$ over the field $F$ is a set $Q$, together with a surjective map $\di\!\!:Q\to \mathcal{D}$ such that:
\begin{enumerate}
\item[(i)] for each $\A \in \D$, the fiber $\di^{-1}(\A)$ has the structure of a one dimensional vector space over $F$,
\item[(ii)] there is a product defined in $Q$ which makes it into an Abelian monoid and the map $\di$ is a monoid homomorphism, that is, for $q, r \in Q$,
\[
\di(qr)=\di(q)\, \di(r),
\]
and
\item[(iii)] the product distributes over the addition in each fiber, that is, for $q, r_1, r_2 \in Q$ with $\di(r_1)=\di(r_2)$,
\[
q (r_1 + r_2) = qr_1 + qr_2,
\]
and the product associates with the product by scalars in the sense of
\[
\alpha (qr) = (\alpha q) r,
\]
where $q,r \in Q$ and $\alpha \in F$.
\end{enumerate}
The rank of $Q$, ${\rm rank}(Q)$, is the rank of its group of dimensions.
\end{definition}

This structure can be thought of as an algebraic fiber bundle, where the base structure is the group $\D$ and where over each element of it we place a fiber which is a one dimensional vector space.  The fibers are not independent, for they have algebraic bounds given by the condition of the projection map being a monoid homomorphism.  All fibers are isomorphic as vector spaces, and isomorphic to the field $F$, but there is one fiber of particular interest: the fiber $\di^{-1}(1_{\D})$, the set of quantities of dimension one.  The identity element in $Q$ is denoted $1_Q$ and, since $\di$ is a homomorphism, necessarily $\di(1_Q)=1_{\D}$, so $1_Q$ is a quantity of dimension one, as expected.  Therefore, there is a natural isomorphism between the field $F$ and the fiber of quantities of dimension one, that assigning the number 1 in $F$ with $1_Q$.  For this reason, this fiber can be identified with $F$ when needed.

\subsection{System of units}\label{subsec:units}
We now turn to the task of defining system of units.  It has been noticed that a system of units is nothing but a choice of a nonzero quantity of each dimension, that is, a basis in each fiber. Remember that a system of units is called coherent if the product of the units of any two quantities $q$ and $r$ gives the unit in the system for the quantity $qr$.  The tool for a precise definition is the concept of section, which is a map which chooses one, and only one, element in $Q$ from each fiber.
\begin{definition}\label{def:section}
A section of the space of quantities $Q$ is a map $\sigma:\D\to\ Q$ such that $\di \circ \sigma = {\rm id}_\D$.  A section is called coherent if the map is a group homomorphism.  The zero section, denoted $\sigma_0$, is the section which selects the zero element of each fiber.  A nonzero section is a section none of which images is a zero element.
\end{definition}
Then we have the following definition.
\begin{definition}\label{def:system-units}
A system of units in a space of quantities is a nonzero section of it.  The system is called coherent if the section is coherent.
\end{definition}

Before proceeding further a word on the zeros of $Q$ is worth mentioning. Since each fiber has a zero element there are many zeros in the space $Q$, all of which constitute $\sigma_0(\D)$, the image of the zero section.  In this construction each zero has a dimension, so $0 \,{\rm m}{\rm s}^{-1}$ is a different quantity than $0 \,{\rm kg}$.

Therefore, rather than speaking of the zero element, in this structure we have to speak of a zero element to refer ourselves to any of these elements in the image of the zero section.  Despite of it, when no confusion is possible we write $q=0$ to symbolize that the quantity $q$ is a zero, without stating explicitly its dimension.  Nevertheless, these zeros behave as is expected from an ordinary zero: the product of a quantity with a zero is a zero, as can be easily verified.  However, it must be noticed that there is nothing in the definition of a space of quantities  to prevent the existence of zero divisors, i.e. nonzero quantities $q$ and $r$ such that their product $qr$ is a zero.  As an extreme example consider a space of quantities with a product defined as $qr=0$ for any dimensionful quantities $q$ and $r$; it satisfies all the axioms of definition~\ref{def:space-of-quantities}.  Zero divisors, if any, are by no means isolated for, if $q$ is a zero divisor, then $\alpha q$, with $\alpha \in F$ is also a zero divisor, so the entire fiber of $q$ is made of zero divisors.  Also if $s$ is another quantity such that $sq$ is not zero, then $sq$ is another zero divisor.  Of course zero divisors do not show up in spaces of quantities of actual measurements, therefore in what follows we only consider spaces of quantities free of zero divisors.  Some advantages we gain from that are collected in next proposition.
\begin{proposition}\label{prop:no-zero-divisors}
In a space of quantities the following properties are equivalent:
\begin{enumerate}
\item There are no zero divisors.
\item The cancellation law holds for the product.
\item Every nonzero quantity is invertible.
\item There exists a coherent system of units.
\end{enumerate}
\end{proposition}
\begin{proof}
We show the first property to be equivalent to each of the other three.  First, if $Q$ is free of zero divisors and $qr_1 = qr_2$ for quantities $q\not= 0$, $r_1$ and $r_2$, then $\di(qr_1)=\di(qr_2)$, so $\di(r_1)=\di(r_2)$ because the cancellation law holds in $\D$.  Then it makes sense to write $0=qr_1-qr_2=q(r_1-r_2)$.  Since $q$ is not a zero, the absence of zero divisors implies $r_1-r_2=0$, so $r_1=r_2$ as desired.  For the other way around consider $q$ and $r_1$ are nonzero quantities such that $qr_1=0$ and choose another quantity $r_2$ in the same fiber as $r_1$.  Then $qr_1=qr_2=0$, but $r_1\not= r_2$, so the cancellation law does not hold.

Now let $q$ be a nonzero quantity and assume $Q$ is free of zero divisors. Let $\tilde q$ be a nonzero quantity in the inverse fiber of $q$, meaning $\di(\tilde q)=\di(q)^{-1}$. Therefore $q\tilde q$ is nonzero and dimensionless and, thus, there is a nonzero scalar $\alpha$ such that $q\tilde q=\alpha 1_Q$.  The quantity $s=\alpha^{-1} \tilde q$ satisfies $qs=1_Q$.  On the contrary, if $q$ and $r$ are nonzero quantities such that $qr=0$ and there is an inverse for $q$, say $s$, then $r=r1_Q=r(qs)=0$, a contradiction.

Finally, assume again $Q$ is free of zero divisors and define a section $\sigma:\D\to Q$ by assigning a nonzero element in the corresponding fiber to each element in a basis of $\D$ and the rest of elements by asking $\sigma$ to be a homomorphism.  The absence of zero divisors assures that $\sigma$ is a nonzero, in addition to coherent, section. On the contrary, assume now that $q$ and $r$ are nonzero elements of $Q$ such that $qr$ is a zero, and let $\sigma$ be a coherent section of $Q$.  Assume $\sigma(\di(q))$ and $\sigma(\di(r))$ are nonzero.  Then they are of the form $\sigma(\di(q))=\alpha q$ and $\sigma(\di(r))=\beta r$ for some nonzero $\alpha$ and $\beta$ in $F$. We have $\sigma(\di(qr))=(\alpha q) (\beta r)=\alpha\beta\, qr$ which is a zero.  Hence, a coherent section is not nonzero, so there exits no coherent system of units.
\end{proof}
The following example is the actual way in which quantity calculus is handled everyday.  Moreover, in section~\ref{sec:homomorphism} we justify that, to some extent, this is the only example of a space of quantities free of zero divisors.
\begin{example}\label{ex:standard-example}
For a field $F$ and a finitely generated free Abelian group $\D$, the set $F \times \D$ together with the projection map $\di\!\!:F\times\D \to \D$ which projects onto the second component, and the operations
\begin{eqnarray*}
(\alpha,\A)+(\beta,\A) = (\alpha+\beta, \A),\\
\beta (\alpha, \A) = (\beta \alpha, \A),\\
(\alpha, \A) (\beta, \B) = (\alpha \beta, \A \B),
\end{eqnarray*}
for any $\alpha, \beta \in F$ and $\A, \B \in \D$, becomes a space of quantities free of zero divisors. A coherent system of units is given from a group homomorphism $\chi:\D\to F^*$, where $F^*$ denotes the multiplicative group of the field, by $\sigma:\D\to F\times\D:\A \mapsto (\chi(\A),\A)$.
\end{example}

The first goal of a formalization of quantity calculus is to justify within the formalism the actual way in which operations between quantities are performed, that is, with the aid of a system of units and operating with the numerical values and with the units separately.  The following paragraphs do this.  Let us start by writing down the expression of any quantity in Maxwell's form.  Let $q$ be a quantity in a space $Q$ free of zero divisors, and let $\sigma$ be a system of units in $Q$.  The dimension of $q$ is $\di(q)$ and the unit in its fiber is $\sigma(\di(q))$. Now, since the latter is not a zero, by proposition~\ref{prop:no-zero-divisors} it has an inverse so we can define the map $\nu:Q \to F$, in which we make use of the identification of the field $F$ with the fiber $\di^{-1}(1_{\D})$ of quantities of dimension one, by
\begin{equation}\label{eq:map-nu}
\nu(q) = q \, \sigma(\di(q))^{-1}.
\end{equation}
The quantities $q$ and $\sigma(\di(q))$ have the same dimension, so the product in equation (\ref{eq:map-nu}) gives a quantity of dimension one which, after identification with an element of $F$, can be regarded as a number: the numerical value of $q$ with respect to the unit $\sigma(\di(q))$. Then we have
\begin{equation}\label{eq:Maxwell2}
q = \nu(q)\, \sigma(\di(q)),
\end{equation}
where we identify $\nu(q)$ with $\{q\}$ and $\sigma(\di(q))$ with $[q]$ as given in equation (\ref{eq:Maxwell}).  In other words, the symbols $\{\cdot\}$ and $[\cdot]$ are nothing but the maps $\{\cdot\} = \nu$ and $[\cdot]= \sigma \circ \di$.  Let us study their algebraic properties.
\begin{proposition}\label{prop:map-square}
In a space of quantities free of zero divisors, the map $[\cdot]:Q\to Q$ verifies
\begin{enumerate}
\item[(i)] for $q_1$, $q_2$ quantities in the same fiber, and $\alpha$, $\beta$ in $F$
\[
[\alpha q_1 + \beta q_2] = [q_1] = [q_2],
\]
\item[(ii)] it is a homomorphism with respect to the product of quantities if and only if $\sigma$ is a coherent section.
\end{enumerate}
\end{proposition}
\begin{proof}
Both items stem directly from the splitting of $[\cdot]$ as the composition $\sigma \circ \di$.  For the first one, since $q_1$ and $q_2$ are in the same fiber, then so is $\alpha q_1 + \beta q_2$, so it is clear that $\di(\alpha q_1 + \beta q_2)= \di(q_1) = \di(q_2)$ and, therefore, the same applies to map $[\cdot]$.  For the second item if $\sigma$ is a group homomorphism, then $[\cdot]$ is the composition of two homomorphism with respect to the product, so it is also a homomorphism.  For the other way around, if $[\cdot]$ is a homomorphism, so is $\sigma$ because the map $\di$ is surjective.
\end{proof}

\begin{proposition}\label{prop:map-nu}
In a space of quantities free of zero divisors, the map $\{\cdot\}=\nu$ defined by equation (\ref{eq:map-nu}) is
\begin{enumerate}
\item[(i)] an $F$-linear homomorphism and
\item[(ii)] a homomorphism with respect to the product of quantities if and only if $\sigma$ is a coherent section.
\end{enumerate}
\end{proposition}
\begin{proof}
For the first item consider two quantities $q_1$ and $q_2$ in the same fiber and two scalars $\alpha$ and $\beta$ in the field $F$ and compute $\{\alpha q_1 + \beta q_2\} = (\alpha q_1 + \beta q_2)\, \sigma(\di(\alpha q_1 + \beta q_2))^{-1}$.  Since $\di(\alpha q_1 + \beta q_2)=\di(q_1)=\di(q_2)$ as in previous proposition, the former expression can be written as $\alpha q_1 \sigma(\di(q_1))^{-1} + \beta q_2 \sigma(\di(q_2))^{-1}$, that is, $\alpha \{q_1\} + \beta \{q_2\}$.

In the second item the if part is trivial.  For the only if part consider $\A$ and $\B$ in $\D$ and choose two nonzero quantities $q$ and $r$ such that $\A=\di(q)$ and $\B=\di(r)$.  Then $\sigma(\A\B) = qr\, \{qr\}^{-1}$ for, since $qr$ is not a zero, then $\{qr\} \not= 0$.  Now, because $\{\cdot\}$ is a homomorphism with respect to the product, the previous expresion gives $q\{q\}^{-1}\, r \{r\}^{-1}= \sigma(\A)\,\sigma(\B)$.
\end{proof}
These propositions set the condition to operate with quantities in the usual way: for quantities $q_1$ and $q_2$ in the same fiber and scalars $\alpha$ and $\beta$
\[
\{\alpha q_1 + \beta q_2 \} =  \alpha \{q_1\} + \beta \{q_2\}; \quad [\alpha q_1 + \beta q_2 ] = [q_1] = [q_2],
\]
and for any quantities $q$ and $r$, only in case of a coherent system of units,
\[
\{q r \} = \{q\} \{r\}; \quad [q r] = [q] [r].
\]

\section{New spaces from old ones}\label{sec:new-spaces}

\subsection{Subspace}\label{subsec:subspace}
\begin{definition}\label{def:subspace}
A subset $S$ of a space of quantities $Q$ is a subspace if, with the operations of $Q$ and the restriction of the projection map, it is a space of quantities.
\end{definition}
Since the projection map $\di$ restricted to $S$ is the projection map of $S$, its image, $\di(S)$, must be a subgroup of $\D$.  Fortunately, it is a well known result of group theory that a subgroup of a free Abelian group is itself free Abelian \cite[Theorem 10.17]{Rot95}.  This is equivalent to the condition of $S$ being closed under the product of quantities.  In particular, $1_{\D}$ is in this subgroup.  The subset $S$ is also  closed under addition of quantities of the same fiber and product by scalars so, if $s$ is a nonzero quantity in $S$, then $\alpha s$, for any scalar $\alpha$, is also in $S$. In other words, the complete fiber containing $s$ is in $S$, as should be for the fibers in $S$ must be one dimensional vector spaces.  This observation rules out a fiber in $S$ containing only the zero element.  As a consequence, the fiber of dimensionless quantities is contained in $S$.  Thus, we have characterized the subspaces of a space of quantities as follows.
\begin{proposition}\label{prop:characterization-of-subspaces}
Let $Q$ be a space of quantities with projection map $\di\!\!:Q\to\D$.  A subset $S \subset Q$ is a subspace if and only if it is of the form $S=\di^{-1}(\E)$, where $\E$ is a subgroup of the group of dimensions $\D$.
\end{proposition}

Trivially $Q=\di^{-1}(\D)$ is a subspace, arising from the improper subgroup of $\D$, and so is the fiber $\di^{-1}(1_{\D})$, the subspace arising from the trivial subgroup of $\D$.  Some nontrivial examples follow.
\begin{example}\label{ex:subspaces}
The space $Q_{\rm geom}$ is a subspace of $Q_{\rm kin}$, for $Q_{\rm geom} = \di^{-1}(\langle \LL \rangle)$, and $\langle \LL \rangle$ is a subgroup of $\D_{\rm kin}$.  Analogously, $Q_{\rm kin}$ is a subspace of $Q_{\rm mech}$, which in turn is a subspace of $Q_{\rm phys}$.

But also $\di^{-1}(\langle \LL^2\rangle)$ is a subspace of $Q_{\rm geom}$ with group of dimensions $\langle \LL^2 \rangle$.
\end{example}

\subsection{Tensor product}\label{subsec:tensor}
From the examples one intuitively expects to be able to build the space $Q_{kin}$ of kinematics quantities from $Q_{geom}$ and $Q_{time}$, the spaces of quantities of geometry and time, respectively.  The technique is somewhat similar to the tensor product of linear spaces so we adopt the notation.  Let $Q$ and $R$ be spaces of quantities free of zero divisors, over the field $F$ and with groups of dimensions $\D_Q$ and $\D_R$ respectively, and projection maps $\di_Q$ and $\di_R$.

In the set $Q\times R$ define the element $(q_1,r_1)$ to be related to $(q_2,r_2)$ if there is $\alpha \in F$ such that $q_1=\alpha q_2$ and $r_2=\alpha r_1$ or such that $q_2=\alpha q_1$ and $r_1=\alpha r_2$.  It is straightforward, though tedious, to check it is an equivalence relation.  The quotient set is denoted $Q \otimes R$ and the equivalence class of the element $(q,r)$ is denoted $q \otimes r$.  Notice $\alpha q \otimes r = q \otimes \alpha r$, in particular $q \otimes 0 = 0 \otimes 0 = 0 \otimes r$.

We now define a structure of space of quantities in $Q \otimes R$.  Its group of dimension is the direct product of $\D_Q$ and $\D_R$, which is a free Abelian group.  The projection map is defined by $\di(q\otimes r)=(\di_Q(q), \di_R(r))$, which is well defined because all the elements in the class $q\otimes r$ have the same image under $\di$.  Define a product in $Q\otimes R$ by $(q_1 \otimes r_1)(q_2\otimes r_2) = (q_1q_2) \otimes (r_1r_2)$, which is independent of the representatives chosen, is commutative and associative and has an identity element: the class $1_Q \otimes 1_R$.

Define the product of the scalar $\gamma \in F$ times $q\otimes r$ by $\gamma (q\otimes r)=(\gamma q) \otimes r = q \otimes (\gamma r)$.  Finally, define the addition of two elements in the same fiber $q_1 \otimes r_1$ and $q_2 \otimes r_2$ in the following manner.  From the absence of zero divisors and proposition~\ref{prop:no-zero-divisors} there are nonzero $q \in Q$ and $r \in R$ such that $q_i = \alpha_i q$ and $r_i = \beta_i r$ for some $\alpha_i,\, \beta_i  \in F$, $i\in\{1,2\}$; define $q_1\otimes r_1 + q_2\otimes r_2 = (\alpha_1\beta_1 + \alpha_2\beta_2) (q\otimes r)$.  The addition and product by scalars in the set of elements of a fiber satisfy the properties of a vector space over $F$ and, moreover, this vector space is of dimension one, for, if $q$ and $r$ are nonzero elements and $\alpha$ and $\beta$ are arbitrary scalars, then $(\alpha q) \otimes (\beta r) = (\alpha \beta) (q \otimes r)$.  The zero element in each fiber is $0 \otimes 0$.

Finally, it is also straightforward to see the projection map behaves well under the product: $\di((q_1\otimes r_1)(q_2\otimes r_2))=\di(q_1\otimes r_1)\di(q_2\otimes r_2)$.  Then we have the following result.
\begin{proposition}\label{prop:tensor-product}
The set $Q\otimes R$, together with the operations defined above, is a space of quantities over the field $F$ with group of dimensions $\D_Q \times \D_R$ and rank ${\rm rank}(Q) + {\rm rank}(R)$.
\end{proposition}

The spaces $Q$ and $R$ can be identified, respectively, with $Q \otimes \di_R^{-1}(1_{\D_R})$ and $\di_Q^{-1}(1_{\D_Q}) \otimes R$, which are subspaces of $Q \otimes R$.

\begin{example}
As announced before, we have $Q_{kin} = Q_{geom} \otimes Q_{time}$.  Also $Q_{mech} = Q_{kin}\otimes Q_{mass}$.
\end{example}

\subsection{Quotient space}\label{subsec:quotient}
The quotient space is a construction intended to reduce the rank of a space of quantities by identifying certain quantities of different dimensions. The quotient cannot be taken with respect to a subspace, but another kind of subset of $Q$, namely, that given by a subsection.
\begin{definition}\label{def:subsection}
Let $Q$ be a space of quantities free of zero divisors.  A subsection of $Q$ is the restriction of a nonzero coherent section $\sigma:\D\to Q$ to a subgroup $\E$ of $\D$.
\end{definition}
Its image $\Sigma = \sigma(\E)$, which is also called subsection for brevity, is the intersection of the subspace $\di^{-1}(\E)$ and $\sigma(\D)$, the image of the section.  Notice that $1_Q \in \Sigma$ because the section $\sigma$ is coherent.  With the aid of $\Sigma$ we can define an equivalence relation in $Q$.  The quantity $q_2$ is equivalent modulo $\Sigma$ to the quantity $q_1$ if $q_2=q_1 s$ for some quantity $s$ in $\Sigma$.  It is reflexive for, as noticed before, $1_Q \in \Sigma$.  It is symmetric, because $s$ is not a zero and, by the absence of zero divisors, it is invertible in $Q$ and $q_1=q_2 s^{-1}$, where $s^{-1}$ is in $\Sigma$ because the section is coherent.  Finally, if $q_2=q_1 s$ and $q_3=q_2 s'$ for $s$ and $s'$ in $\Sigma$ then $q_3=q_1 s s'$, so $q_3$ is equivalent modulo $\Sigma$ to $q_1$ since $s s'$ is in $\Sigma$ by the coherence of $\sigma$.

The quotient set of this equivalence relation is denoted $Q/\Sigma$ and its elements, the equivalence classes, are of the form $q\Sigma$, which denotes the set of the elements $qs$ with $s$ running in $\Sigma$.  We now provide the quotient set with suitable operations to convert it into a space of quantities.  First we describe its group of dimensions.  Since equivalent elements $q_1$ and $q_2=q_1 s$ are identified in the quotient set, their dimensions must be identified as well. The obvious candidate for the group of dimensions is, thus, the quotient group $\D/\E$.  In such a case it is only natural to define the projection map, $\widehat \di\!\!:Q/\Sigma\to\D/\E$ by making the following diagram commutative.
\[
\begin{CD}
Q @>{\rho}>> Q/\Sigma \\
@V{\di}VV @VV{\widehat\di}V \\
\D @>{\hat\rho}>> \D/\E
\end{CD}
\]
where the maps $\rho$ and $\hat\rho$ are the natural projections of each set into its respective quotient set.  In other words, $\hat\rho \circ \di = \widehat \di \circ \rho$.  Unfortunately, the quotient of a free Abelian group is not necessarily free Abelian and, hence, $\D/\E$ does not necessarily qualify as a group of dimensions.  Therefore, though the algebraic structure is well defined, the subsection must be carefully chosen so as the subgroup $\E$ makes the quotient $\D/\E$ a free Abelian group.  For instance, in the group of dimensions of kinematics quantities, $\D_{\rm kin}$ generated by $\LL$ and $\TT$, the subgroup $\E_1=\langle \LL \rangle$ gives $\D/\E_1 \cong \langle \TT\rangle$, which is free Abelian, while the subgroup $\E_2=\langle \LL^2 \rangle$ gives $\D/\E_2 \cong \langle \LL \rangle / \langle \LL^2 \rangle  \times \langle \TT\rangle$, which is not free Abelian.  From now on we assume that $\E$ is chosen so as to make $\D/\E$ free Abelian.

The product in $Q/\Sigma$ is defined by the rule $\big ( q_1 \Sigma\big) \big ( q_2 \Sigma\big) = (q_1 q_2) \Sigma$ which is easily checked to be independent of representatives.  We have to check the condition which links the product and the projection map, but $\widehat\di\big( q_1\Sigma q_2\Sigma \big)=\widehat\di\big( (q_1q_2) \Sigma\big)=\widehat\di \circ \rho (q_1q_2)$ by the definition of the product in $Q/\Sigma$ and the definition of $\rho$.  Now, by the commutativity of the diagram and because both, $\di$ and $\hat \rho$, are homomorphisms, the latter expression equals $\hat\rho \circ \di (q_1q_2)=\hat\rho \circ \di(q_1) \hat\rho \circ \di(q_2)= \widehat\di(q_1 \Sigma) \widehat\di(q_2\Sigma)$, so we conclude that $\widehat\di$ is a monoid homomorphism.

The product with a scalar $\alpha$ from the field $F$ is defined by $\alpha \big ( q \Sigma \big ) = (\alpha q) \Sigma$, which is also independent of the choice of representative $q$ in the class $q \Sigma$. For the addition notice that if $q_1 \Sigma$ and $q_2 \Sigma$ are elements in the same fiber in $Q/\Sigma$, i.e. $\widehat\di(q_1\Sigma) = \widehat\di(q_2\Sigma)$, its sum cannot be defined simply as $(q_1 + q_2) \Sigma$, because $q_1$ and $q_2$ need not be in the same fiber in $Q$.  We only know $\di(q_2) = \di(q_1) \A$ for some $\A \in \E$.  Denote $s=\sigma(\A)$, an element in $\Sigma$, and define $q_1'=q_1 s$, so $q_1 \Sigma = q_1' \Sigma$, hence $\di(q_1') = \di(q_1)\di(s) = \di(q_2)$, so they are in the same fiber in $Q$.  Now we can define the addition as $\big ( q_1 \Sigma \big ) + \big( q_2 \Sigma \big ) = (q_1' + q_2) \Sigma$.  We could have taken instead an equivalent element of $q_2$ in the fiber of $q_1$ getting the same result. In the fiber of $q\Sigma$, the zero element is the class $q_0\Sigma$, where $q_0$ is the zero in the fiber of $q$, and is formed by the zeros of the fibers of $Q$ represented in the class $q\Sigma$.

It is straightforward to check that the conditions of definition~\ref{def:space-of-quantities} hold for $Q/\Sigma$, so we state the result as follows.
\begin{proposition}\label{prop:quotient}
If $\E$ is a subgroup of $\D$ such that $\D/\E$ is free Abelian then the set $Q/\Sigma$, together with the operations defined above, is a space of quantities with group of dimensions $\D/\E$ and rank given by ${\rm rank}(\D)-{\rm rank}(\E)$.
\end{proposition}

The mechanism of taking quotients is the algebraic tool underlying what is common practice in Physics of choosing ``systems of units" such that some specified universal constants become dimensionless and take on the numerical value 1, as is shown in the following examples. But it has to be remarked that the mechanism goes beyond a change of system of units; it is indeed a change of space of quantities.
\begin{example}\label{ex:reduction-of-units}
In the space of quantities of kinematics, $Q_{\rm kin}$, of rank 2, the group of dimensions can be generated as well by $\TT$ and $\LL\TT^{-1}$, so we write $\D_{\rm kin}=\langle \TT, \LL\TT^{-1} \rangle$. The speed of light, $c$, is a quantity in this space with $\di(c)=\LL\TT^{-1}$. Consider the subgroup $\E=\langle \LL\TT^{-1} \rangle$ and a nonzero coherent section $\sigma$ such that $\sigma(\LL\TT^{-1}) = c$.  The subsection obtained by the restriction of $\sigma$ to $\E$ defines the subset of the quantity $c$ and all its powers: $\Sigma = \{\dots c^{-2}, c^{-1}, 1_{Q_{\rm kin}}, c, c^2, \dots\}$.  Finally, the quotient space $Q_{\rm kin}/\Sigma$ is made of the classes of kinematics quantities modulo $\Sigma$, with group of dimensions $\D_{\rm kin}/\E = \langle T\E \rangle$, free Abelian of rank 1.  That is, all the quantities in the quotient space have dimensions of time and its powers, the fiber of quantities of dimension one contains $c$ and all its powers which are equivalent to the identity $1_{Q_{\rm kin}}$.  This is common practice, for instance, when dealing with Special Relativity, where all kinematics quantities are measured as time.
\end{example}
\begin{example}\label{ex:Plank's-units}
Consider the construction underlying Planck's units, an instance of natural units.  In the space of quantities of Physics, $Q_{\rm phys}$, of rank 5, the group of dimensions was presented in example~\ref{ex:groups-of-dimensions} as generated by $\LL, \TT, \MM, \I, \Theta$.  We take a quotient such that the universal constants $c$, the speed of light, $h$, Planck's constant, $G$, the gravitational constant, $k_{\rm C}$, Coulomb's constant and $k_{\rm B}$, Boltzmann's constant, are set to 1.   It is straightforward to check that the five constants are dimensionally independent, that is, their dimensions $\di(c) = \LL\TT^{-1}$, $\di(h)=\LL^2\TT^{-1}\MM$, $\di(G)=\LL^3\TT^{-2}\MM^{-1}$, $\di(k_{\rm C})=\LL^3\TT^{-2}\MM\I^{-2}$ and $\di(k_{\rm B})=\LL^2\TT^{-2}\MM\Theta^{-1}$ are independent in $\D_{\rm phys}$ and, moreover, the set
\[
\{\LL\TT^{-1}, \LL^2\TT^{-1}\MM, \LL^3\TT^{-2}\MM^{-1}, \LL^3\TT^{-2}\MM\I^{-2}, \LL^2\TT^{-2}\MM\Theta^{-1} \}
\]
is a basis of the group.  Now the coherent section defined on this five generators by
\[
\begin{array}{l}
\sigma(\LL\TT^{-1}) = c, \\
\sigma(\LL^2\TT^{-1}\MM) = h, \\
\sigma(\LL^3\TT^{-2}\MM^{-1}) = G, \\
\sigma(\LL^3\TT^{-2}\MM\I^{-2}) = k_{\rm C}, \\
\sigma(\LL^2\TT^{-2}\MM\Theta^{-1}) = k_{\rm B}, \\
\end{array}
\]
is clearly nonzero, so it defines a subsection (when restricted to the improper subgroup $\E=\D_{\rm phys}$).  The quotient space $Q_{\rm phys}/\Sigma$ is a space of quantities of rank 0 since its group of dimensions is the trivial group. Hence, all the quantities are dimensionless and the five aforementioned constants are equivalent to the identity quantity.
\end{example}

\section{Homomorphism of spaces of quantities. Isomorphic spaces}\label{sec:homomorphism}

In this section the tool for comparison of spaces of quantities is defined and its properties studied.  The goal is the classification of spaces of quantities, which is achieved in theorem~\ref{th:classification}.
\begin{definition}\label{def:homomorphism}
Let $Q$ and $R$ be spaces of quantities over the field $F$.  A map $\psi:Q\to R$ is a homomorphism of spaces of quantities if
\begin{enumerate}
\item[(i)] for any two quantities $q_1$, $q_2$ in $Q$
\[
\psi(q_1q_2) = \psi(q_1) \psi(q_2),
\]
that is, it is a monoid homomorphism with respect to the product, and
\item[(ii)] if $q_1$ and $q_2$ are quantities in the same fiber of $Q$, then $\psi(q_1)$ and $\psi(q_2)$ are in the same fiber in $R$ and
\[
\psi(\alpha q_1 + \beta q_2) = \alpha \psi(q_1) + \beta \psi(q_2), 
\]
for $\alpha$ and $\beta$ in $F$, so $\psi$ is a linear map in each fiber.
\end{enumerate}
\end{definition}

The homomorphism $\psi$ induces a group homomorphism between the base groups, $\D_Q$ and $\D_R$. If $\di_Q$ and $\di_R$ are the respective projection maps, define the map $\phi:\D_Q\to\D_R$ so that the following diagram commutes,
\[
\begin{CD}
Q @>{\psi}>> R \\
@V{\di_Q}VV @VV{\di_R}V \\
\D_Q @>{\phi}>> \D_R
\end{CD}
\]
that is, $\di_R \circ \psi = \phi \circ \di_Q$.  It is well defined because $\di_Q$ and $\di_R$ are surjective and $\psi$ preserves fibers and it is straightforward to check that $\phi$ is a group homomorphism. 

The map $\phi$ says what fibers of $Q$ are mapped into each fiber of $R$.  As an example, if $q$ is a quantity of dimension one in $Q$ then $\di_R(\psi(q))=\phi(\di_Q(q))=\phi(1_{\D_Q})=1_{\D_R}$, so the fiber of quantities of dimension one in $Q$ is mapped to the fiber of quantities of dimension one in $R$.  

An isomorphism $Q\to R$ is a bijective homomorphism and define $Q$ and $R$ as isomorphic spaces, denoted $Q\cong R$. The basic maps relating the spaces of quantities defined so far, i.e. the identity map ${\rm id}_Q:Q\to Q$, the inclusion map $i:Q\hookrightarrow R$, where $Q$ is a subspace of $R$, and the natural projection map $\rho:Q \to Q/\Sigma$, where $\Sigma$ is the image of a subsection of $Q$, are all homomorphisms of spaces of quantities.  The trivial homomorphism is the map which sends every element in $Q$ to the dimensionless zero of $R$.  By a zero homomorphism we understand a homomorphism in which all the elements in $Q$ are mapped to zero elements in $R$, such as the trivial map, but there are other zero homomorphisms, as many as group homomorphisms between $\D_Q$ and $\D_R$.  For given such a group homomorphism $\phi$, which in turn defines which fibers in $Q$ are mapped to which fibers in $R$, it is enough to define $\psi:Q\to R$ by sending each $q \in Q$ to the zero element of the fiber assigned by $\phi$.

In fact it is necessary to understand the behaviour of fibers and zeros under a homomorphism.  It is clear that the image of a zero is a zero.  If $q$ is a nonzero element of $Q$ but $\psi(q)$ is a zero in $R$, then all the fiber of $q$ is mapped to the same zero, for $\psi(\alpha q)=\alpha \psi(q)$ which is the same zero for any $\alpha \in F$.  On the other hand, if $\psi(q)$ is not a zero, then the fiber of $q$ is mapped isomporphically (as vector spaces) to the fiber of $\psi(q)$.  In particular, if $\psi(1_Q)$ is zero, then the homomorphism is a zero homomorphism, for $\psi(q) = \psi(q 1_Q) = \psi(q) \psi(1_Q)$ which is a zero for any $q$.  This expresion also proves that if $\psi(1_Q)$ is not a zero then it is $1_R$, the identity in $R$.

So far a homomorphism can be defined by setting which fiber of $R$ is the image of each fiber of $Q$ and by setting which fibers of $Q$ are mapped to zero and which of them are mapped isomorphically to their corresponding fibers.  In the case of interest of spaces free of zero divisors the result can be improved.
\begin{proposition}\label{prop:homomorphism-without-zero-divisors}
Let $Q$ be a space of quantities free of zero divisors and $\psi:Q\to R$ a homomorphism of spaces of quantities.  Then $\psi$ is a zero homomorphism if and only if $\psi(1_Q)$ is the dimensionless zero.
\end{proposition}
\begin{proof}
The if part has already been proved.  Assume now that $\psi(1_Q)$ is not a zero and let $q$ be a nonzero element of $Q$ which, thus, has an inverse $q^{-1}$.  Since $\psi(1_Q)=\psi(qq^{-1})=\psi(q) \psi(q^{-1})$ is not a zero, we conclude that $\psi(q)$ is not a zero.
\end{proof}
Of course, only the nonzero homomorphisms are of interest for us to be able to compare spaces of quantities, so from now on we only consider this kind of homomorphisms.  The following are basic properties of homomorphisms of spaces of quantities.
\begin{proposition}\label{prop:properties-homomorphisms}
Let $\psi:Q\to R$ be a nonzero homomorphism of spaces of quantities.  Then:
\begin{enumerate}
\item[(i)] the image of a subspace of $Q$ is a subspace of $R$,
\item[(ii)] the preimage of a subspace of $R$ is a subspace of $Q$,
\item[(iii)] the preimage of a section of $R$ is a section of $Q$ and
\item[(iv)] the preimage of a subsection of $R$ is a subsection of $Q$.
\end{enumerate}
\end{proposition}
\begin{proof}
Let $S$ be a subspace of $Q$, which is characterized, by proposition~\ref{prop:characterization-of-subspaces}, as $S=\di_Q^{-1}(\E)$ for a subgroup $\E$ of $\D_Q$.  The projection of its image is $\di_R(\psi(S))=\phi(\di_Q(S))=\phi(\E)$, which is a subgroup of $\D_R$. Since $\psi$ is a nonzero homomorphism, every fiber in $S$ is mapped onto a fiber in $\psi(S)$, so we conclude that $\psi(S)$ coincides with $\di_R^{-1}(\phi(\E))$, so it is a subspace of $R$.

Consider now $S$ to be a subspace of $R$.  Its inverse image $\psi^{-1}(S)$ is made of the fibers which are mapped into $S$.  But these fibers are given by $\di_Q^{-1}(\phi^{-1}(\di_R(S)))$.  Since $\di_R(S)$ is a subgroup of $\D_R$, so is $\phi^{-1}(\di_R(S))$ with respect to $\D_Q$ and, thus, $\psi^{-1}(S)$ is a subspace of $Q$.

Let $\sigma$ be a section of $R$. Since all fibers of $R$ are represented in the section, it is clear that its inverse image, $\psi^{-1}(\sigma(\D_R))$, contains at least an element from each fiber in $Q$.  We now show that there is no more than one from each fiber.  Assume $q_1$ and $q_2$ are elements in the same fiber in $Q$ with $\psi(q_1)$ and $\psi(q_2)$ in the section of $R$. Then $\psi(q_1)$ and $\psi(q_2)$ belong to the same fiber in $R$, which means $\psi(q_1)=\psi(q_2)$ because they are in a section.  Therefore, since fibers in $Q$ are mapped isomorphically to fibers in $R$, this leads to $q_1=q_2$.

Finally, considering a subsection as the intersection of a subspace and a section in $R$ it is clear that the inverse image of such intersection is a intersection of a subspace and a section in $Q$, thus, a subsection.
\end{proof}

The kernel of a nonzero homomorphism $\psi:Q\to R$ is defined as $\ker \psi = \psi^{-1}(1_R)$.  Since $\{1_R\}$ is a subsection of $R$, its inverse image defines, by proposition~\ref{prop:properties-homomorphisms}, a subsection in $Q$.   The image of a homomorphism, $\im \psi = \psi(Q)$ is, by the same proposition, a subspace of $R$. The kernel and the image so defined satisfy an isomorphism theorem.
\begin{theorem}\label{th:isomorphism}
Let $\psi:Q\to R$ be a nonzero homomorphism of spaces of quantities.  Then
\[
Q/{\ker \psi} \cong \im \psi
\]
as spaces of quantities.
\end{theorem}
\begin{proof}
The first step is to check that the quotient of the theorem is indeed a space of quantities. Let us denote by $K$ the kernel of $\psi$, which is a subsection of $Q$.  Thus we only have to show the group of dimensions of $Q/K$ is free Abelian and, to that end, we have to identify the projection of $K$ on $\D_Q$. We claim this projection to be precisely the kernel of the induced group homomorphism: $\di_Q(K)=\ker \phi$.  Let $q$ be in $K$.  Then $\phi \circ \di_Q (q) = \di_R \circ \psi (q) = 1_{\D_R}$, so $\di_Q(q) \in \ker \phi$ which shows one inclusion. Now let $\A$ be in $\ker \phi$, and let $q$ be a nonzero element in the fiber $\di_Q^{-1}(\A)$.  Then $1_{\D_R}=\phi(\A)=\di_R \circ \psi(q)$, so $\psi(q)$ has dimension one and can be written as $\psi(q)=\alpha 1_R$ for a nonzero $\alpha$ in $F$ (since $\psi$ is a nonzero homomorphism).  Consider the element $q_1=\alpha^{-1} q$.  Then $\psi(q_1)=1_R$, so $q_1 \in K$ and $\di_Q(q_1)=\di_Q(q)=\A$, so $\A \in \di_Q(K)$, which shows the other inclusion and the claim is proved.  The group of dimensions of $Q/K$ is, thus, $\D_R/\ker \phi$, which is isomorphic with $\im \phi$ by the isomorphism theorem for groups.  Since $\im \phi$ is a subgroup of $\D_R$ it is free Abelian, and so is $\D_R/\ker \phi$.  Therefore $Q/K$ is a space of quantities.

The rest of the proof is standard. Define the map $\hat \psi:Q/K\to \im \psi$ by $\hat\psi(qK)= \psi(q)$.  It is straightforward to check, first, it is well defined; second, it is a homomorphism of spaces of quantities; third, it is a bijection.
\end{proof}
\begin{example}\label{ex:speed-light-homomorphism}
Let us revisit example~\ref{ex:reduction-of-units} from the viewpoint of homomorphisms.  Consider the map $\psi:Q_{\rm kin}\to Q_{\rm time}$ given in the following form: the image of the quantity called \emph{second} be itself, while the image of the speed of light, $c$, be $1_{Q_{\rm time}}$.  Then $\ker \psi = \{1_{Q_{\rm kin}}\} \cup \{c^n \,:\, n \in \mathbb{Z}\}$ and $\im \psi = Q_{\rm time}$. Theorem~\ref{th:isomorphism} says that $Q_{\rm kin}/\ker \psi$ is isomorphic with $Q_{\rm time}$.
\end{example}

The next result is the classification theorem for spaces of quantities free of zero divisors.
\begin{theorem}\label{th:classification}
Two spaces of quantities over the same field, free of zero divisors, are isomorphic if and only if they have the same rank.
\end{theorem}
\begin{proof}
First consider two spaces of quantities $Q \cong R$.  Then there is an isomorphism $\psi:Q\to R$ which induces a group homomorphism $\phi:\D_Q\to \D_R$.  We only need to show the latter to be an isomorphism for it is a well known result of the theory of free Abelian groups that two such groups are isomorphic if and only if they have the same rank \cite[Theorem 10.14]{Rot95}.  But that is obvious since the map $\phi$ is nothing but the rule which says which fiber in $Q$ is mapped to what fiber in $R$ and, since $\psi$ is an isomorphism, this mapping of fibers is a bijection.

Second, assume $Q$ and $R$ are two spaces of quantites of the same rank, that is, their groups of dimensions $\D_Q$ and $\D_R$ have the same rank.  Therefore there is a group isomorphism $\phi:\D_Q\to \D_R$ and it defines a bijection of the fibers in $Q$ with the fibers in $R$.  If we can assign a linear isomorphism between each pair of fibers we are done.  To that end it is enough to map a nonzero element of each fiber in $Q$ with a nonzero element of its corresponding fiber in $R$.  Now, since both $Q$ and $R$ are free of zero divisors, by proposition~\ref{prop:no-zero-divisors} each of them has a coherent system of units, say $\sigma_Q$ and $\sigma_R$ respectively.  Define a map $\psi:Q\to R$ by giving its action on the set $\sigma_Q(\D_Q)$ so that the following diagram commutes,
\[
\begin{CD}
Q @>{\psi}>> R \\
@A{\sigma_Q}AA @AA{\sigma_R}A \\
\D_Q @>{\phi}>> \D_R
\end{CD}
\]
and extend it linearly in each fiber.  This map is easily seen to be an isomorphism of spaces of quantities, so $Q\cong R$.
\end{proof}
As a last example, we show that, up to isomorphism, the example \ref{ex:standard-example} is the only space of quantities over a group of dimensions and a field free of zero divisors.
\begin{example}
Let $Q$ be a space of quantities over the field $F$ with group of dimensions $\D$ and free of zero divisors.  Let $\sigma$ be a coherent system of units and $\nu$ the map defined in equation (\ref{eq:map-nu}).  Then the map $\psi:Q\to F\times\D$ given by $\psi(q)=(\nu(Q),\di(q))$ is an isomorphism of spaces of quantities.  Its inverse is $\psi^{-1}(\alpha, \A)=\alpha \, \sigma(\A)$.
\end{example}
This is to say that every space of quantities $Q$, free of zero divisors, is isomorphic with $F\times\D$, that is, all the quantities can be written as in equation (\ref{eq:Maxwell}). But the isomorphism is not canonical, for it depends on the system of units chosen.

\bibliographystyle{unsrt}      
\bibliography{alvaro-p-raposo}

\end{document}